\documentclass[12pt]{article}
\usepackage[usenames]{color}
\usepackage[colorlinks=true,
linkcolor=webgreen,
filecolor=webbrown,
citecolor=webgreen]{hyperref}

\definecolor{webgreen}{rgb}{0,.5,0}
\definecolor{webbrown}{rgb}{.6,0,0}
\usepackage{fullpage}
\usepackage{graphicx}
\usepackage{amscd}
\usepackage{color}
\usepackage{float}
\usepackage{graphics,amsmath,amssymb}
\usepackage{amsthm}
\usepackage{amsfonts}
\usepackage{latexsym}
\usepackage{epsf}

\newtheorem{theorem}{Theorem}

\newtheorem{example}{Example}

\begin{document}
\begin{center}
{\LARGE\bf Ramanujan-Type Series Related to
 Clausen's Product}

\vspace{.25in}
\large 
John M. Campbell\\York University\\
Toronto, ON \\
Canada\\
\href{mailto:maxwell8@yorku.ca}{\tt maxwell8@yorku.ca}
\end{center}

\begin{abstract}
  Infinite series are evaluated through the manipulation of a series for $\cos(2t \sin^{-1}x)$ resulting from Clausen's Product.
  Hypergeometric series equal to an expression involving  $\frac{1}{\pi}$ are determined. Techniques to evaluate generalized hypergeometric
 series are discussed through perspectives of experimental mathematics. 
\end{abstract}

\section{Introduction}

  Clausen's product is discussed in the fascinating book by Borwein, Bailey, and Girgensohn, \emph{Experimentation in Mathematics: Computational Paths to Discovery}
 \cite{bbg}. New Ramanujan-type series derived from Clausen's product such as the following are given in this paper: 
$\sum_{n=2}^{\infty}   (  \frac{1}{4})^{3n-1}   \frac{(2n-3)!(4n+3)!}{((2n)!(n+1)(2n+1))^{2}(n-2)!n!}$; the following section ({\bf Ramanujan-Type Series})
 explains the derivation of this series and a related series. We discuss the evaluation of hypergeometric series through integral substitution. For example, we shall
 soon consider an infinite array of hypergeometric series of the form  $_{3}   F _{2}    (\frac{1}{2},  \frac{2k}{n},\frac{2n-2k}{n};\frac{3}{2},2;1)$. Furthermore,
 we discuss methods to incorporate the Riemann-zeta function into hypergeometric-like series, and discuss related new hypergeometric generalizations, evaluating
hypergeometric series such as $_{6}   F _{5}    (1,1,\frac{9}{4},\frac{5}{2},\frac{11}{4},3;2,\frac{37}{16},\frac{41}{16},\frac{45}{16},\frac{49}{16};1)$.

  We are interested in the evaluation of infinite series. There has been much work discussing hypergeometric identities of the form 
$_{p}   F _{q}  (a_{1},a_{2},...,a_{p};b_{1},b_{2},...,b_{q};x)$ for constant $x$ such that $a_{1}$,$a_{2}$,...,$a_{p}$ and/or $b_{1}$,$b_{2}$,...,$b_{q}$ contain a
 variable/variables. There does not seem to be as much work on hypergeometric series of the form $_{p}   F _{q}  (a_{1},a_{2},...,a_{p};b_{1},b_{2},...,b_{q};x)$
 for variable $x$, which can be manipulated through integration with respect to $x$ (and by replacing $x$ by $f(x)$ and integrating). For example, only series of
 the form $_{p}   F _{q}  (a_{1},a_{2},...,a_{p};b_{1},b_{2},...,b_{q};x)$ for constant $x$ are discussed in \cite{ira}  (which is cited in \cite{ab}). 

 ``Common", modern computational software is usually unable to directly evaluate hypergeometric series of the form 
$_{p}   F _{q}  (a_{1},a_{2},...,a_{p};b_{1},b_{2},...,b_{q};x)$ for variable $x$; we thus consider the evaluation of such series to be important. Consider infinite
 series of the form $\sum g(n) (1-x^a)^{bn} x^{dn} $, where $g(n)$ represents a combination of a finite number of elementary functions (and $a$, $b$, and $d$ are
 positive integers). Integration of the expression equal to such series is often very difficult, yielding complicated results, some of which we shall here present.

  The evaluation of a series such as $\sum_{k=0}^{\infty} \frac{4^k}    {\binom{2k}{k}(2k+1)(k+1)^{2}}$ is an example of a tangible approach to exploring series
 representations of important constants through integral substitution. Although the theorems below have somewhat more cumbersome derivations, briefly describing our
 evaluation of the aforementioned series will hopefully serve as a straightforward introduction to this paper. The following well-known and important series (for
 $x^{2} \leq 1$) is given in \cite{gradsh}: $ (\sin^{-1} x )^{2} = \sum_{k=0}^{\infty}  \frac{4^{k} (k!)^{2}  x^{2k+2}}  {(2k+1)!(k+1)}$. Divide both sides of this
 equation by $x$ and integrate:

\begin{equation*}
\int \frac{(\sin^{-1} x)^2}{x}    dx  =    \frac{\operatorname{Li}_{3}(e^{2i \sin^{-1}x})}{2}  + i \sin^{-1}x\operatorname{Li}_{2}(e^{2i \sin^{-1}x}) + \frac{1}{3} i (\sin^{-1}  x)^{3}   +  (\sin^{-1}  x)^{2}  \ln (1-e^{2i \sin^{-1}x})
\end{equation*}

  Give the above integral an upper limit of one and a lower limit of zero, resulting in an integral evaluated in \cite{jchoi} (described as ``an interesting definite
integral"). It is now clear that $\sum_{k=0}^{\infty} \frac{4^k}    {\binom{2k}{k}(2k+1)(k+1)^{2}} = \frac{\pi^2}{2} \ln (2) - \frac{7 \zeta (3)}{4}$. The following
related series is elegantly proven in \cite{borweinchamber}: $\sum_{n=1}^{\infty}  \frac{4^n}{n^{3}\binom{2n}{n}} = \pi^{2} \ln(2) - \frac{7 \zeta (3)}{2}$.

  Let us continue by providing the infinite array of $_{3} F _{2}$ series we have in mind, which is based upon Clausen's product. Consider:

\begin{equation*}
(_{2}   F _{1}    (\frac{k}{n},\frac{n-k}{n};\frac{3}{2};\sin ^{2}  x))^2
     = (  \frac{      \csc x \sin      (      (-1 + 2 (1 - \frac{k}{n})) \sin   ^{-1  }  (\sin x)      )   }{-1 + 2 (1 - \frac{k}{n})}     ) ^2
\end{equation*}

\begin{theorem}%
\begin{eqnarray*}
\lefteqn{\int_{0}^{\frac{\pi}{2}}  (  \frac{      \csc x \sin                   (      (-1 + 2 (1 - \frac{k}{n})) \sin   ^{-1  }  (\sin x)      )   }{-1 + 2 (1 - \frac{k}{n})}     ) ^2   dx  = }    \\
  &  &   -  \frac{1}{4 (-2 k + n)}    i e^{ \frac{2 i k \pi}{n}   }   n (  \frac{n}{2 k - n}  +   \frac{e^{\frac{4 i k \pi}{n}} n}{2 k - n}     -     2 e^{ \frac{2 i k \pi}{n} } \pi \cot ( \frac{2 k \pi}{n}) -        \psi_{0} (\frac{1}{2} - \frac{k}{n}) + \psi_{0}( 1 - \frac{k}{n})    \\ 
  &  &   +     e^{\frac{4 i k \pi}{n}} \psi_{0} (-\frac{1}{2} + \frac{k}{n}) -     e^{\frac{4 i k \pi}{n}}    \psi_{0}(\frac{k}{n}))  
\end{eqnarray*}

$=$

\begin{equation*}
 _{3}   F _{2}    (\frac{1}{2},  \frac{2k}{n},\frac{2n-2k}{n};\frac{3}{2},2;1)  \frac{\pi}{2}
\end{equation*}
\end{theorem}%

\begin{proof}

  Given Clausen's product, the above theorem holds.

\end{proof}

  Perhaps it is probable that the above theorem has been considered before.

   Integrating    $ ( _{2} F _{1}    (  2 - \frac{k}{n}  , \frac{k}{n}  ;  \frac{5}{2}     ;    \sin ^{2} x  ))^2$      proves to be much more cumbersome, involving
 multifarious hypergeometric series.

  Now, let us briefly discuss some more simple results.

  Clausen's product can be used to obtain the following relationships \cite{bbg}:

\begin{equation}
\label{eq:cos2t}
\sum_{n=0}^{\infty} \frac{(t)_{n}(-t)_{n}}    {(2n)!} (2x)^{2n} = \cos(2t \sin^{-1}x)
\end{equation}

\begin{equation*}
-\frac{1}{2}  \sum_{n=1}^{\infty}   \frac{(t)_{n}(-t)_{n}}   {(2n)!}  ( 4\sin^{2} x)^{n}  =  \sin^{2} (tx)
\end{equation*}

  A generalization of (\ref{eq:cos2t}) which does not involve the Pochhammer symbol is given in \cite{jolley}.

  Consider integrating (\ref{eq:cos2t})  for $\cos(\sin^{-1}(\sqrt{x}))$:

\begin{equation}
\label{eq:2/3}
\sum_{n=0}^{\infty}   \frac{(\frac{1}{2})_{n}(-\frac{1}{2})_{n}}    {(2n)!(n+1)} 4^{n} x^{n+1} = -\frac{2}{3} (1-x)^{\frac{3}{2}}+\frac{2}{3}
\end{equation}

Consider integrating (\ref{eq:2/3}):

\begin{equation}
\label{eq:2/5}
\sum_{n=0}^{\infty}    \frac{(\frac{1}{2})_{n}(-\frac{1}{2})_{n}}        {(2n)!(n+1)(n+2)}   4^{n} x^{n+2}
=-\frac{2}{3}   ( - \frac{2}{5}   (1-x)^{\frac{5}{2}}   -x        )    - \frac{4}{15}
\end{equation}

  (\ref{eq:2/3}) can be transformed as follows:

\begin{equation}
\label{eq:5/8}  
\sum_{n=0}^{\infty}    \frac{(\frac{1}{2})_{n}(-\frac{1}{2})_{n}}    {(2n)!(n+1)(2n+3)}   4^{n} x^{2n+3}
=-\frac{2}{3}    ( \sqrt{1-x^2}  (\frac{5x}{8}-\frac{x^3}{4})  -x+\frac{3}{8} \sin^{-1} x     )
\end{equation}

  Using such techniques, it is not difficult to evaluate series such as $\sum_{n=0}^{\infty}     \frac{(\frac{1}{6})_{n}(-\frac{1}{6})_{n}}       {(2n)!(2n+3)}$
or  $\sum_{n=0}^{\infty}   \frac{(\frac{1}{8})_{n}(-\frac{1}{8})_{n}}        {(2n)!(n+2)}    $ . Simple variations of series  such as (\ref{eq:2/3})  often lead to
 outlandish results.

  Consider multiplying both sides of (\ref{eq:5/8}) by $x$ and integrating:

\begin{equation}
\label{eq:144}  
\sum_{n=0}^{\infty}    \frac{(\frac{1}{2})_{n}(-\frac{1}{2})_{n}}       {(2n)!(n+1)(2n+3)(2n+5)}   4^{n} x^{2n+5}
=\frac{1}{144}    (  32 x^{3}  -18 x^{2}    \sin^{-1}  (x)    +\sqrt{1-x^2}   (4x^{2} -16x^{2} -3)x + 3 \sin^{-1}  x   )   
\end{equation}

  The Ramanujan-type series which follow are proven using the following very important definite integral:

\begin{equation}
\label{eq:extremelybeautiful}  
\int_{0}^{\frac{\pi}{2}}   \sin^{2n} t dt = \frac{\pi}{2}   \frac{(\frac{1}{2})_n}{n!}
\end{equation}

  Consider substituting (\ref{eq:extremelybeautiful}) into (\ref{eq:cos2t}) and variations of (\ref{eq:cos2t}), yielding fairly simple series, the partial sums of
 many of which are known. For example consider the partial sums of 
$\sum_{n=0}^{\infty}   4^{n}   \frac{(\frac{1}{32})_{n}(-\frac{1}{32})_{n}(\frac{1}{2})_{n}}   {n!(2n)!}$ or
 $ \sum_{n=2}^{\infty}  ( \frac{1}{4})^{2n-1}    \frac{\binom{2n-3}{n-1}\binom{2n+1}{n}}        {n(n+1)(n+2)}$.

  Relationships such as (\ref{eq:144}) can be manipulated using (\ref{eq:extremelybeautiful}), resulting in series which are perhaps new;
even if such series are not \emph{outright} new, some computational software is unable to evaluate such series, and they do not appear in
authoritative mathematical literature on the topic.

 The following integral is worthy of mention:

\begin{equation}
\label{eq:alsobeautiful}  
\int_{0}^{\pi} x \sin^{n} x dx = \frac{\pi^{\frac{3}{2}} \Gamma (\frac{n+1}{2})   }      {2  \Gamma (1+ \frac{n}{2})}
\end{equation}

  Now, consider using relationships such as the following to manipulate series such as  (\ref{eq:2/3}):

\begin{equation}
\label{eq:23vari}  
\int_{0}^{\pi}   x (  \frac{2}{3} - \frac{2}{3}   (1 - \sin x)^{\frac{3}{2}})   dx  =  \frac{\pi}{9}   (20-16 \sqrt{2} + 3 \pi)
\end{equation}

  Let us begin by  acknowledging that it seems to us that the following relationship is highly important, at least in terms of the evaluation
of hypergeommetric series:

\begin{equation*}
\sum_{n=1}^{\infty} (((1 - x^a)^{(b  n) + c}) (n + g)^{d} f^n = f (1 - x^a)^{b + c} \Phi(f (1 - x^a)^{b}, -d, 1 + g)
\end{equation*}

 So, the above series can be transformed into a hypergeometric series. If  $\operatorname{Re} (a) > 0 $ and  $\operatorname{Re}(c + b n) > -1$, then:

\begin{equation}
\label{eq:notgeneral}  
\int_{0}^{1} (((1 - x^a)^{(b  n) + c}) (n + g)^{d} f^n   dx   =  f^{n} (g + n)^{d}   \frac{\Gamma(1 + \frac{1}{a}) \Gamma(1 + c + b n)}{\Gamma(1 + \frac{1}{a} + c + b n)} 
\end{equation}%

  The above relationships lead to an extremely important question, which is author is presently unable to answer:
  {\bf Is it possible to determine a meaningful, general evaluation of the following integral?}

\begin{equation*}
\int_{0}^{1}  f (1 - x^a)^{b + c} \Phi(f (1 - x^a)^{b}, -d, 1 + g) dx
\end{equation*}%

  The above integral leads to very beautiful evaluations of hypergeometric series.

 Questions related to the above question such as the following arise:   ``Is it possible to meaningfully evaluate $\int_{0}^{\frac{\pi}{2}}  \sin^{-1} \sin^{2} x   dx$?"

  Given that  $\sum_{n=2}^{\infty} \frac{(1-x^a)^{bn}}{c^n} = -    \frac{(1 - x^a)^{2 b}  }{c (-c + (1 - x^a)^b)}$ ,
is it possible to determine a meaningful, general evaluation of $\int_{0}^{1} -    \frac{(1 - x^a)^{2 b}  }{c (-c + (1 - x^a)^b)}         dx $? ``General" evaluations of this integral for only a \emph{single} variable are
very complicated: consider {\bf Theorem  17}.

  Given that $\sum_{n=2}^{\infty} \frac{(1-x^a)^{bn}}{n^c} = -(1 - x^a)^{b} + \operatorname{Li}_{c}((1 - x^a)^b)$, is it possible to determine a meaningful, general
 evaluation of $\int_{0}^{1} -(1 - x^a)^{b} + \operatorname{Li}_{c}((1 - x^a)^b)        dx $?  These questions seem to us to be highly important. In \cite{weiss} it
 is stated that ``No general algorithm is known for integration of polylogarithms of functions."  Such general series can be used to evaluate generalized
 hypergeometric series $_{p} F _{q}$ for relatively arbitrary $p$ and $q$.

  More generally, compared to  (\ref{eq:notgeneral}),  we are interested in integrals of the form $\int^p_q (f(x))^{n} g(x) dx$ which
 are equal to a finite number of combinations of elementary functions, where $f(x)$ and $g(x)$ represent a finite number of combinations of elementary functions
 such that ${n} \in \mathbb{Z}^+$;  such integrals can be used to establish new hypergeometric series. Let $h(x) = (f(x))^{n} g(x)$. Consider the following
 straightforward two steps:%

1) If possible, evaluate $\sum_{n=1}^{\infty} h(x)j(n)$, where $j(n)$ represents a finite combination of elementary functions (e.g. $j(n)=\frac{n^{45}+1}{n+2}$)%

2) Integrate the above sum (i.e. using $\int^p_{q} h(x) dx$) and its equivalent expression%

  Let the above simple two-step procedure be called the HJ-algorithm.
  In this paper, we present some general tendencies of the HJ-algorithm, through perspectives of experimental mathematics.
As stated in the excellent \emph{Mathematics by Experiment: Plausible Reasoning in the 21st Century} \cite{bo}, \begin{quote} The computer provides the mathematician
with a ``laboratory" in which he or she can perform experiments: analyzing examples, testing out new ideas, or searching for patterns.\end{quote}

  Most of the theorems in this paper are not given rigorous analyses; indeed, most of the series are presented for their own sake. Although the ideas (i.e. theorems)
 being ``tested" in this paper are new in the sense that they may not have been previously published or widely recognized, the HJ-algorithm is so simple that it can
hardly be described as a ``new idea". Finally, although we present some seemingly unpredictable tendencies of the HJ-algorithm, we presently do not claim to
have considerable knowledge of general patterns in the hypergeometric series. Indeed, some modern computational software is unable to evaluate 
extremely important integrals such as $\int_0^{1} \frac{(1 - x^a)^b}{-c + (1 - x^a)^b} dx$.

  Many of the evaluations of the hypergeometric series of the forms
we analyze in this paper involve special functions such as the polylogarithm function; for the time being, we mostly discuss hypergeometric series which have a
 ``nice" (although often long and intricate) form consisting  of a finite number of combinations of elementary functions; evaluating an infinite series (i.e.
 a hypergeometric series) using another infinite series (i.e. the polylogarithm function) seems somwhat redundant. Having said that, we present a way to incorporate
 the Riemann-zeta function into certain forms of hypergeometric series.
 
  We will now present a variety of infinite series.   {\bf Theorem 2} is particularly worthy of note, because it is most likely to be original, and does not
 seem to be equal to a single hypergeometric function.

\section{Ramanujan-Type Series}

\begin{theorem}%
\begin{equation}
\sum_{n=2}^{\infty}   (  \frac{1}{4})^{3n-1}   \frac{(2n-3)!(4n+3)!}{((2n)!(n+1)(2n+1))^{2}(n-2)!n!}  =  -\frac{553}{96} - \frac{42 \ln(\frac{2-\sqrt{2}}{2+\sqrt{2}}) -34\sqrt{8}}             {9\pi}
\end{equation}
\end{theorem}%

\begin{proof} 

  {\bf Theorem 2} spawns from (\ref{eq:2/3}). Determine $\int (1-\sin^{4} x)^{\frac{3}{2}} dx$ and substitute (\ref{eq:extremelybeautiful}) into (\ref{eq:2/3}).

\end{proof} 

\begin{theorem}%
\begin{equation}
_{3} F _{2}   ( - \frac{5}{2}, \frac{1}{4},  \frac{3}{4}   ;   \frac{1}{2}   ,   1   ;   1    ) =
   \frac{1141}{960 \sqrt{2} \pi}    -     \frac{103 \ln   ( \frac{2 - \sqrt{2}}{2 + \sqrt{2}}   )}{ 256 \pi}  
\end{equation}
\end{theorem}%

\begin{proof} 

  Use the same technique as in {\bf Theorem 2}, using (\ref{eq:2/5}).

\end{proof} 

  Use similar techniques to prove the following very similar equation:

\begin{equation*}
_{3}  F _{2}    (-\frac{5}{2},\frac{1}{4},\frac{3}{4};\frac{3}{2},2;1) =  \frac{1067}{6720 \sqrt{2} \pi}        +   \frac{383 \tanh^{-1}(\frac{1}{\sqrt{2}}) }{128 \pi}     
\end{equation*}

  Consider how zeta-type sums can be determined using (\ref{eq:extremelybeautiful}). In the great book \emph{Computation Techniques for the Summation of Series}
 \cite{sofo}, Anthony Sofo provides the following series, which results from integrating $\frac{1}{2} \sum_{n=1}^{\infty} \frac{(2x)^{2n}}{n^{2} \binom{2n}{n} }$
 repeatedly:

\begin{equation*}
\sum_{n=1}^{\infty}   \frac{(2x)^{2n+3}}   {n^{2}\binom{2n}{n} 4 (2n+3)(2n+2)(2n+1) }
= \frac{x(2x^{2}+3)(\sin^{-1}x)^2}{3} + \frac{2\sqrt{1-x^2}(11x^{2}+4)\sin^{-1}x}{9}-\frac{85x^3}{27}-\frac{8x}{9}
\end{equation*}

  We leave it as an exercise to demonstrate how (\ref{eq:extremelybeautiful}) and the above series can be used to prove the equation
$\sum_{n=1}^{\infty} \frac{1}{n^{2}(n+1)^{2}(n+2)}=\frac{2\pi^{2}-19}{8}$, the partial sums of the series of which are known.

  The above series may be simplified as follows:

\begin{equation*}
   _{3} F _ {2}  (1,1,1;2,\frac{7}{2};x^2)   \frac{x^5}{15} = \frac{x(2x^{2}+3)(\sin^{-1}x)^2}{3} + \frac{2\sqrt{1-x^2}(11x^{2}+4)\sin^{-1}x}{9}-\frac{85x^3}{27}-\frac{8x}{9}
\end{equation*}

  It is worthy of note that the following infinite series, which is reminiscent of the series under discussion, is indicated in \cite{jolley}:

\begin{equation*}
\sum_{n=1}^{\infty}  \frac{(-1)^{n} \theta^{2n-1} (-m)_{2n-1} }{(2n-1)!}   =  \sin   (m \tan^{-1} \theta)   (1+\theta^2)^{\frac{m}{2}}
\end{equation*}

   Consider incorporating alternating harmonic numbers into similar
such series: $\sum_{j=0}^{\infty} \frac{\binom{2j}{j}}{4^{j} (2j+1)}H^{'}_{2n+1} = -2C + \frac{3\pi \ln 2}{2} $. Consider using the integrals
$\int_{0}^{1} \frac{x^{2n+1}}{x+1} dx = H^{'}_{2n+1} - \ln 2$ and $\int_{0}^{1} \frac {\sin^{-1}x}{x+1}  dx = -2C+\pi \ln 2$ to establish the above
equation.

  It can be established that:

\begin{equation*}
\int_{0}^{1} x^{2n} (1-x^2)^{n}   dx  =   \frac{4^n}{(2n+1)\binom{4n+1}{2n}}
\end{equation*}

  Integrals of such forms are very important identities. We are presently interested in series involving the expression $x^{2n} (1-x^2)^{n}$. In
the excellent \emph{ Some New Formulas for} $\pi$ \cite{almkv}, integrals of the form $\int x^{pn} (1-x)^{(m-p)n} dx$ are analyzed.
 In \cite{ramanujan1}, the equation $ \int_{0}^{1}  x^{m} (1-x)^{n}  dx    = \frac{\Gamma (m+1)  \Gamma   (n+1)}  {\Gamma (m+n+2 }       $ is given,
and is described as ``...an extremely well-known formula for the beta-function... ." Let us consider the integral $\int_{0}^{1} x^{2n} (1-x^2)^{n}   dx$
 and similar integrals. We intend to establish new infinite series, and classificationss of integrals (which are similar to the beta function).

\begin{theorem}%
\begin{equation*}
_{4} F _ {3}   (1,1,1,\frac{3}{2};\frac{7}{4},2,\frac{9}{4};\frac{1}{4}) = \frac{15}  {8} (-64+\pi(8+\pi)+(\cos^{-1}2)^{2}+8\sqrt{3}\cosh^{-1}(7)+\cosh^{-1}(26)\ln(2-\sqrt{3}))
\end{equation*}
\vspace{.11in}
\end{theorem}%

\begin{proof}

\begin{equation*}
\sum_{n=1}^{\infty}  \frac{x^{2n} (1-x^2)^n}{n^2} = \operatorname{Li}_{2} (x^{2}-x^{4})
\end{equation*}

  Integrate both sides of the above equation.

\end{proof}

\begin{example}%
\begin{equation*}
_{3} F _{2}  (\frac{3}{2},2,2;\frac{7}{4},\frac{9}{4};\frac{1}{4}) = \frac{15}{36}  (6- \sqrt{3} \ln(2+\sqrt{3}) )
\end{equation*}
\end{example}%

\begin{example}%
\begin{equation*}
\sum_{n=1}^{\infty}    \frac{4^{n}(n+1)^2}         {(2n+1)\binom{4n+1}{2n}} 
= \frac{-36+45\pi+10\sqrt{3}\cosh^{-1}2}  {144}
\end{equation*}
\end{example}%

  It can be established that:

\begin{equation*}
\int_{0}^{1}x^{4n+1}  (1-x^4)^{2n+\frac{1}{2}}   dx =  \pi   \frac{(2n-1)!!(4n+1)!!}   {2^{3n+3}(3n+1)!}
\end{equation*}

  The above equation can be proven using the gamma function of half-integer expressions. It is not difficult to consider an infinitude of integrals of the form
$\int_{0}^{1} x^{a} (1-x^b)^{c}$  equal to an expression involving the gamma function of half-integer expressions.

  Consider:

\begin{equation*}
\sum_{n=2}^{\infty}    (\frac{1}{32})^{n}   \frac{\binom{2n}{n} \binom{2n-3}{n-2} (2n-1)}{n}
= \frac{(\Gamma (\frac{1}{4})) ^2}{8\pi^{\frac{3}{2}}} - \frac{9}{32}
\end{equation*}

  This series is determined through the complete elliptic integral of the first kind, being based upon $_{2} F _{1}   (\frac{1}{2},\frac{1}{2};2;-1)$.
 In \emph{Theory and Problems of Complex Variables with an Introduction to Conformal Mapping and its Applications} \cite{spiegel}
it is indicated that:

\begin{equation*}
\int_{0}^{\frac{\pi}{2}}  \frac{1}{ \sqrt{1-\frac{\sin^{2}x}{2}}}  dx  = \frac{(\Gamma(\frac{1}{4}))^2}{4 \sqrt{\pi}}
\end{equation*}

  Use the binomial theorem as follows:

\begin{equation*}
\frac{1}{\sqrt{1-x}} = -\sum_{n=0}^{\infty}   x^{n}   \frac{(-\frac{1}{2})_{n} (2n-1)}{n!}
\end{equation*}

  Substitute (\ref{eq:extremelybeautiful}):

\begin{equation*}
\frac{1}{\sqrt{1-\frac{\sin^{2}x}{2}}}   = \sum_{n=0}^{\infty}   \sin^{2n} x \frac{(-\frac{1}{2})_{n}(2n-1)}{2^{n}n!}
\end{equation*}

\section{Computational Perspectives}

  Let us discuss some results, some of which are possibly original (given that computational software is unable to evaluate them), from perspectives
of experimental mathematics.

\begin{example}%
\begin{equation*}
_{3} F_ {2}   (\frac{1}{2},\frac{2}{3},\frac{4}{3};\frac{3}{2},\frac{3}{2};1)   =   \frac{3\sqrt{3}  \coth^{-1}  \sqrt{3} }{2}
\end{equation*}
\end{example}%

\begin{proof}

  Consider:

\begin{equation*}
_{2} F_{1} (\frac{1}{3},\frac{2}{3};\frac{3}{2};z^2) ^{2} = (\frac{3\sin(\frac{\sin^{-1} z}{3}   )}{z} )^2
\end{equation*}

  Integrate both sides of the above equation

\end{proof}

\begin{example}%
\begin{equation*}
_{3} F _{2}   (\frac{1}{3},\frac{1}{2},\frac{5}{3};\frac{3}{2},\frac{3}{2};1) = \frac{9}{4} - \frac{3}{4} \sqrt{3} \coth^{-1}\sqrt{3}
\end{equation*}
\end{example}%

\begin{example}%
\begin{equation*}
_{3} F _{2}   (\frac{1}{4},\frac{1}{2},\frac{7}{4};\frac{3}{2},\frac{3}{2};1)  = \frac{1}{9}  (12 \sqrt{2} - 3 \pi  i  + 12 \tanh^{-1}(-1)^{\frac{3}{4}})
\end{equation*}
\end{example}%

\begin{example}%
\begin{eqnarray*}
\lefteqn{  _{3} F _{2}  (\frac{1}{8},\frac{1}{2},\frac{15}{8};\frac{3}{2},\frac{3}{2};1) =  }   \\ 
  &  &     \frac{ 16 \sqrt{2 + \sqrt{2}}}{21}   + \frac{2 i \pi}{7} -   \frac{8}{7} i \tan^{-1}((-1)^{\frac{1}{8}}) + \frac{4}{7} i \sqrt{2} \tan^{-1}(1 - \sqrt{2})     \\  
  &  &    -   \frac{4}{7} i \sqrt{2} \tan^{-1}(1 + \sqrt{2}) -   \frac{4}{7} i \sqrt{2} \tan^{-1}(1 - (-1)^{\frac{1}{8}} \sqrt{2})      \\  
  &  &    +   \frac{4}{7} i \sqrt{2} \tan^{-1}(1 + (-1)^{\frac{1}{8}} \sqrt{2})
\end{eqnarray*}
\end{example}%

\begin{example}%
\begin{equation*}
_{3}   F  _{2}   (\frac{1}{6},\frac{1}{2},\frac{11}{6};\frac{3}{2},\frac{3}{2};1)   =  
\frac{9 \sqrt{3}}{10}   + \frac{\sqrt{3} \pi i}{10}   -    \frac{3  \sqrt{3} \tanh^{-1}( \frac{3 + 2 i \sqrt{3}}{7}    )}{5} 
\end{equation*}
\end{example}%

  We consider integrals of the form $\int_{0}^{1}  (\frac{\sin(a \sin^{-1}x)}{x} )^{2}  dx$   to be important.

  Consider series for $\frac{1}{\pi}$ using the above latter technique presented in hypergeometric form:

\begin{theorem}%
\begin{equation*}
_{3} F_{2} (\frac{1}{2},\frac{2}{3},\frac{4}{3}; \frac{3}{2}, 2;1) = \frac{3\sqrt{3}(3-\ln 4)}{2 \pi}
\end{equation*}
\end{theorem}%

\begin{proof}

Consider:

\begin{equation*}
(_{2} F_{1}   ( \frac{1}{3}     , \frac{2}{3} ;  \frac{3}{2} ; \sin^{2} z   ))^2 = (3 \csc (z)   \sin (\frac{1}{3} \sin^{-1} (\sin z)))^2
\end{equation*}

  Integrate both sides of the above equation.

  Recall (\ref{eq:extremelybeautiful}).

\end{proof}

\begin{theorem}%
\begin{equation*}
_{3} F_{2} (\frac{1}{2}, \frac{1}{3}, \frac{5}{3} ; \frac{3}{2},  2  ; 1      ) =  \frac{3 \sqrt{3} (3+ \ln 16 )}  {8  \pi}
\end{equation*}
\end{theorem}%

\begin{proof}

Consider:

\begin{equation*}
(_{2} F_{1}   ( \frac{1}{6}    , \frac{5}{6} ;  \frac{3}{2} ; \sin^{2} z   ))^2 = (\frac{3}{2} \csc (z)   \sin (\frac{2}{3} \sin^{-1} (\sin z)))^2
\end{equation*}

Given the HJ algorithm, {\bf Theorem 6} holds.

\end{proof}

\begin{theorem}%
\begin{equation*}
_{3} F_{2} (\frac{1}{2}, \frac{1}{4}, \frac{7}{4}; \frac{3}{2}, 2 ; 1) = \frac{8   (\sqrt{2}+6\tanh^{-1} \sqrt{\frac{2 - \sqrt{2}}{2 + \sqrt{2}}    } )}{9 \pi}
\end{equation*}
\end{theorem}%

\begin{proof}

\begin{equation*}
(_{2} F_{1}   ( \frac{1}{8}    ,    \frac{7}{8} ;  \frac{3}{2} ; \sin^{2} z   ))^2 = (\frac{4}{3} \csc (z)   \sin (\frac{3}{4} \sin^{-1} (\sin z)))^2
\end{equation*}

Given the HJ algorithm, {\bf Theorem 7} holds.

\end{proof}

\begin{theorem}%
\begin{equation*}
_{3} F_{2} (\frac{1}{2}, \frac{1}{6}, \frac{11}{6}; \frac{3}{2}, 2 ; 1)  = \frac{\frac{18}{25}+\frac{12}{5}\sqrt{3}\coth^{-1}\sqrt{3}}    {\pi}
\end{equation*}
\end{theorem}%

\begin{proof}

\begin{equation*}
(_{2} F_{1}   ( \frac{1}{12}   ,  \frac{11}{12} ;  \frac{3}{2} ; \sin^{2} z   ))^2 = (\frac{6}{5} \csc (z)   \sin (\frac{5}{6} \sin^{-1} (\sin z)))^2
\end{equation*}

Given the HJ algorithm, {\bf Theorem 8} holds.

\end{proof}

 The following very useful relationship is given in ``Some New Formulas for $\pi$" \cite{almkv}:

\begin{equation*}
\frac{1}{\binom{3n}{n}(3n+1)} = \int_{0}^{1} x^{2n} (1-x)^{2n} dx 
\end{equation*}

    Series of the form $\pi = \sum_{n=0}^{\infty}    \frac{S(n)}{\binom{mn}{pn}a^n}$ are discussed in \cite{almkv}.

  Consider another example of the above generalization:

\begin{theorem}%
\begin{equation*}
_{3} F_{2}  (  \frac{1}{2}, \frac{8}{5}, \frac{2}{5}; \frac{3}{2}, 2; 1  ) =  \frac{5 (5 (5 + 3 \sqrt{5}) + 3 (5 + 3 \sqrt{5}) \ln 5 -   12 (5 + 2 \sqrt{5}) \ln(5 - \sqrt{5}) +     6 (5 + \sqrt{5}) \ln (5 + \sqrt{5}))}{9 \pi (1 + \sqrt{5}) \sqrt{2 (5 + \sqrt{5})}}
\end{equation*}
\end{theorem}%

\begin{proof}

Consider: $(_{2} F_{1}   ( \frac{1}{5}   ,  \frac{4}{5} ;  \frac{3}{2} ; \sin^{2} z   ))^2 = (\frac{5}{3} \csc (z)   \sin (\frac{3}{5} \sin^{-1} (\sin z)))^2$.

\end{proof}

  The following integral is given in \cite{gradsh}: $\int_0^{\pi} x \sin^{p} x dx = \frac{\pi^2}{2^{p+1}} \frac{\Gamma(p+1)}{(\Gamma(\frac{p}{2}+1))^2}$.
 We presently evaluate new hypergeometric series using such integrals through manipulations of the left-hand side of the following equation discussed in \cite{bo}:

\begin{eqnarray*}
\lefteqn{_{3}F _{2} (a,2b-a-1,a-2b+2;b,a-b+\frac{3}{2};\frac{x}{4}) = }    \\  
&   &  _{3}F _{2} (\frac{a}{3}, \frac{a+1}{3}, \frac{a+2}{3};b,a-b+\frac{3}{2}; \frac{-27x}{4(1-x)^3}) (1-x)^{-a}
\end{eqnarray*}

  Replace $x$ with $4\sin^{2}(x)$ and integrate. Let $\heartsuit (n,x)$ = $_{2}F_{1} (-n-1,n+2; \frac{3}{2}; \sin^{2}x)$. 

\section{Further Results}

\begin{theorem}%
\begin{eqnarray*}
\lefteqn{_{3} F_{2} (-\frac{6}{5},\frac{1}{2},\frac{11}{5};1,\frac{3}{2};1) = }    \\
  &  & ( 5 (\sqrt{5 (5 + \sqrt{5})} (25 - 6 \ln(-1 + \sqrt{5})) +    \sqrt{5 + \sqrt{5}} (-25 + 6 \ln(-1 + \sqrt{5}))     \\ 
  &  &       +    \sqrt{5 - \sqrt{  5}} (1 + \sqrt{5}) (\ln(64) - 6 \ln(3 + \sqrt{5}))) )/ ( 204 \pi \sqrt{2}   )
\end{eqnarray*}
\end{theorem}%

\begin{proof}

  Consider: $\heartsuit (\frac{1}{5},x) = \frac{5}{17} \csc(x) \sin(\frac{17}{5} \sin^{-1}(\sin x))$. Use (\ref{eq:extremelybeautiful}).

\end{proof}

\begin{theorem}%
\begin{eqnarray*}
\lefteqn{_{3} F_{2} (-\frac{9}{8}, \frac{1}{2}, \frac{17}{8}; 1, \frac{3}{2}; 1)   =  }    \\  
  &  &  ((   \frac{2}{ 117 \pi}  (128 \sqrt{2 - \sqrt{2}} +     9 (\ln(161 - 112 \sqrt{2} - 4 \sqrt{2 (1594 - 1127 \sqrt{2})})    \\  
  &  &  -     2 \sqrt{2} \ln(\sqrt{2} + \sqrt{2 - \sqrt{2}}) +   \sqrt{2} \ln(4 - \sqrt{2} - 2 \sqrt{2 (2 - \sqrt{2})})))))
\end{eqnarray*}
\end{theorem}%

\begin{proof}

  Integrate $\heartsuit (\frac{1}{8},x)$.

\end{proof}

\begin{theorem}%
\begin{equation*}
_{3} F _{2}  (1,1,\frac{3}{2};\frac{5}{3},\frac{13}{6};1) =  \frac{28}{3} +   \frac{ 2^{\frac{1}{3}} 7 \pi}{9 \sqrt{3}}  -   \frac{ 2^{\frac{1}{3}} 14 \tan^{-1}(   \frac{1 + 2^{\frac{2}{3}}}{\sqrt{3}} )}   {3 \sqrt{3}}  +   \frac{14}{9} 2^{\frac{1}{3}}  \ln(2 - 2^{\frac{2}{3}}) -  \frac{7}{9}  2^{\frac{1}{3}}  \ln(2 (2 + 2^{\frac{1}{3}} + 2^{\frac{2}{3}}))
\end{equation*}
\end{theorem}%

\begin{proof}

  Use the HJ-algorithm.

\begin{equation*}
\sum_{n=1}^{\infty} \frac{(1 - x^3)^{2n}}{n} = -\ln(-x^{3} (-2 + x^{3}))
\end{equation*}%

  Integrate both sides of the above equation.

\end{proof}

\begin{example}%
\begin{eqnarray*}
 _{3} F _{2}    (1,1,\frac{3}{2};\frac{19}{12}, \frac{25}{12};1)  & = &   \frac{91}{72} (  12 - 2^{\frac{1}{6}} \sqrt{3} \tan^{-1} (\frac{-1 + 2^{\frac{5}{6}}}{\sqrt{3}}) -  2^{\frac{1}{6}} \sqrt{3} \tan^{-1}(\frac{1 + 2^{\frac{5}{6}}} {\sqrt{3}}) +  2^{\frac{1}{6}} \ln(2 - 2^{\frac{5}{6}})   \\
&   & \mbox{}  + \frac{\ln(2 + 2^{\frac{2}{3}} - 2^{\frac{5}{6}})} {2^{\frac{5}{6}}}  -  2^{\frac{1}{6}} \ln(2 + 2^{\frac{5}{6}}) - \frac{\ln(2 + 2^{\frac{2}{3}} +  2^{\frac{5}{6}})}{2^{\frac{5}{6}}}   ) 
\end{eqnarray*}
\end{example}%

\begin{proof}

  Use the HJ-algorithm.

\begin{equation*}
\sum _{n=1}^{\infty} \frac{(1-x^6)^{2n}}{n} = -\ln(-x^{6}(-2+x^6))
\end{equation*}%

  Integrate both sides of the above equation.

\end{proof}

\begin{example}%
\begin{eqnarray*}
_{4} F _{3}  (1,\frac{5}{4},\frac{7}{4},2;\frac{3}{2},\frac{11}{6},\frac{13}{6};\frac{2}{27})  = \frac{7}{2}(-10+\pi+(-2\sqrt{-7 - i}+\sqrt{2-14i})\tan^{-1}(\sqrt{\frac{-1-i}{2}})     \\   
-2\sqrt{-7+i} \tan^{-1}(\sqrt{\frac{-1+i}{2}})+\sqrt{2+14i}\tan^{-1}(\sqrt{\frac{-1+i}{2}}))
\end{eqnarray*}
\end{example}%

\begin{proof}

  Use the HJ-algorithm.

\begin{equation*}
\sum_{n=1}^{\infty} \frac{(1-x^2)^{n} x^{4n}}{2^n}    =    -  \frac{x^{4} (-1 + x^2)}{2 - x^{4} + x^6}   
\end{equation*}%

  Consider:

\begin{equation*}
\int_0 ^1 (1-x^2)^{n} x^{4n} dx = \frac{\Gamma(1 + n) \Gamma(\frac{1}{2} + 2 n)}{2 \Gamma(\frac{3}{2} + 3 n)}   = \frac       {2^{n} n! (4n-1)!!}    { (6n+1)!!}  
\end{equation*}%

\end{proof}

\begin{example}%
\begin{eqnarray*}
\lefteqn{_{6} F _{5}   (1,1,1,\frac{5}{4},\frac{3}{2},\frac{7}{4}; \frac{11}{8},\frac{13}{8},\frac{15}{8},2,\frac{17}{8};1) =    }      \\
  &  &  \frac{315}{64} (-32 + 3 \pi^{2} -  4 (-1)^{\frac{1}{8}} 2^{\frac{1}{4}} \ln(-1 + (-1)^{\frac{1}{8}} 2^{\frac{1}{4}}) +  2 i \cot^{-1}    (  \frac{2^{\frac{3}{4}} - \sqrt{2 + \sqrt{2}}}{\sqrt{ 2 - \sqrt{2}} }  )  \\
  &  & \ln(-1 + (-1)^{\frac{1}{8}} 2^{\frac{1}{4}})  +   4 (-1)^{\frac{1}{8}} 2^{\frac{1}{4}} \ln(1 + (-1)^{\frac{1}{8}} 2^{\frac{1}{4}}) -  2 i \cot^{-1}   ( \frac{2^{\frac{3}{4}} - \sqrt{2 + \sqrt{2}}}{\sqrt{2 - \sqrt{2}}}  ) \\
  &  & \ln(     i + (-1)^{\frac{1}{8}} 2^{\frac{1}{4}}) +  \ln      ^{2}   (   \frac{1 + (-1)^{\frac{1}{8}} 2^{\frac{1}{4}}}{1 - (-1)^{\frac{1}{8}} 2^{\frac{1}{4}}}   ) -  4 (-1)^{\frac{7}{8}} 2^{\frac{1}{4}} \ln(-1 + (-1)^{\frac{7}{8}} 2^{\frac{1}{4}}) \\
  &  &  +  \ln    ^{2}    (   \frac{1 - (-1)^{\frac{7}{8}} 2^{\frac{1}{4}}}{1 + (-1)^{\frac{7}{8}} 2^{\frac{1}{4}}}     ) +  2 i \cot^{-1}   (  \frac{2^{\frac{3}{4}} + \sqrt{2 + \sqrt{2}}}{\sqrt{  2 - \sqrt{2}}} )    (\ln(-1 + (-1)^{\frac{1}{8}} 2^{\frac{1}{4}})  - \ln(1 + (-1)^{\frac{1}{8}} 2^{\frac{1}{4}})     \\
  &  &  + \ln(-1 + (-1)^{\frac{7}{8}} 2^{\frac{1}{4}}) -    \ln(1 + (-1)^{\frac{7}{8}} 2^{\frac{1}{4}})) +  4 (-1)^{\frac{7}{8}} 2^{\frac{1}{4}} \ln(1 + (-1)^{\frac{7}{8}} 2^{\frac{1}{4}})    \\
  &  &  -  2 i \cot^{-1}      (  \frac{2^{\frac{3}{4}} - \sqrt{2 + \sqrt{2}}}{\sqrt{2 - \sqrt{2}}}    )  \ln(1 + (-1)^{\frac{7}{8}} 2^{\frac{1}{4}}) +  2 \ln(-1 + (-1)^{\frac{7}{8}} 2^{\frac{1}{4}}) \ln(1 + (-1)^{\frac{7}{8}} 2^{\frac{1}{4}})   \\
  &  & -  2 i \pi (   \ln (    \frac{(-1)^{\frac{1}{8}} }{(-1)^{\frac{1}{8}} - 2^{\frac{1}{4}}}     ) -  \ln(-1 + (-1)^{\frac{1}{8}} 2^{\frac{1}{4}}) +  \ln(1 + (-1)^{\frac{1}{8}} 2^{\frac{1}{4}}) +       2 \ln(1 + (-1)^{\frac{7}{8}} 2^{\frac{1}{4}}))      \\
  &  & - 4 \sqrt{2} \ln(-1 + \sqrt{2}) -  2 \ln  ^{2}(-1 + \sqrt{2}) + 4 \sqrt{2} \ln(1 + \sqrt{2}) +    4 \ln(-1 + \sqrt{2}) \ln(1 + \sqrt{2})    \\
  &  &   - 2 \ln(1 + \sqrt{2})^2 +  \ln(3 + 2 \sqrt{2})^2 -  \ln(1 + (-1)^{\frac{7}{8}} 2^{\frac{1}{4}}) \ln(   \operatorname{Im}(\sqrt{1 - i})^2     \\
  &  &   + (-1 + \operatorname{Re}(\sqrt{1 - i}))^2) -  \ln(-1 + (-1)^{\frac{7}{8}} 2^{\frac{1}{4}}) \ln(  \operatorname{Im}(\sqrt{1 - i})^2       \\
  &  &  + (1 + \operatorname{Re}(\sqrt{1 - i}))^2) +  \ln(1 + (-1)^{\frac{7}{8}} 2^{\frac{1}{4}})  \ln(  \operatorname{Im}(\sqrt{1 - i})^2      \\
  &  &    + (1 + \operatorname{Re}(\sqrt{1 - i}))^2) -  \ln(-1 + (-1)^{\frac{1}{8}} 2^{\frac{1}{4}}) \ln( \operatorname{Im}(\sqrt{1 + i})^2      \\
  &  &   + (-1 + \operatorname{Re}(\sqrt{1 + i}))^2) +  \ln(1 + (-1)^{\frac{1}{8}} 2^{\frac{1}{4}})  \ln(   \operatorname{Im}(\sqrt{1 + i})^2     \\
  &  &  + (-1 + \operatorname{Re}(\sqrt{1 + i}))^2) +  \ln(-1 + (-1)^{\frac{1}{8}} 2^{\frac{1}{4}}) \ln(  \operatorname{Im}(\sqrt{1 + i})^2    \\
  &  &  + (1 + \operatorname{Re}(\sqrt{1 + i}))^2) -  \ln(1 + (-1)^{\frac{1}{8}} 2^{\frac{1}{4}}) \ln(  \operatorname{Im}(\sqrt{1 + i})^2  + (1 + \operatorname{Re}(\sqrt{1 + i}))^2))
\end{eqnarray*}
\end{example}%

\begin{proof}

  Use the HJ-algorithm.

\begin{equation*}
\sum_{n=1}^{\infty} \frac{(1 - x^2)^{4 n}}{n^2} = \operatorname{Li}_{2}((-1+x^2)^4)
\end{equation*}

\end{proof}

\begin{example}%
\begin{eqnarray*}
\lefteqn{_{5} F _{4}   (\frac{1}{4},\frac{3}{4},1,1,2;\frac{1}{2},\frac{5}{6},\frac{7}{6},3;\frac{2}{27})     =   }     \\
  &  &  \frac{18}{5} - \frac{491 \pi}{35} + \frac{81 \pi^2}{4} +        9 \ln ^{2}  ( \frac{  (-1)^{\frac{1}{8}} + 2^{\frac{1}{4}}}{(-1)^{\frac{1}{8}} - 2^{\frac{1}{4}}}    )    - (\frac{522}{35} - \frac{124 i}{5})     (-1)^{\frac{1}{8}} 2^{\frac{1}{4}} \ln(-1 + (-1)^{\frac{1}{8}} 2^{\frac{1}{4}})     \\  
  &  &   +    18 \pi i  \ln(-1 + (-1)^{\frac{1}{8}} 2^{\frac{1}{4}}) -      18 i \cot^{-1} ( \frac{-2 + 2^{\frac{1}{4}} \sqrt{2 + \sqrt{2}}}{2^{\frac{1}{4}} \sqrt{2 - \sqrt{2}}}      )     \ln(-1 + (-1)^{\frac{1}{8}} 2^{\frac{1}{4}})      \\      
  &  & +  18 i \cot^{-1}(   \frac{2 + 2^{\frac{1}{4}} \sqrt{2 + \sqrt{2}}}{2^{\frac{1}{4}} \sqrt{2 - \sqrt{2}}}     )  \ln(-1 + (-1)^{\frac{1}{8}} 2^{\frac{1}{4}})     + (\frac{522}{35}  -   \frac{124 i}{5}) (-1)^{\frac{1}{8}} 2^{\frac{1}{4}} \ln(1 + (-1)^{\frac{1}{8}} 2^{\frac{1}{4}})      \\  
  &  &   -      18 i \pi \ln(1 + (-1)^{\frac{1}{8}} 2^{\frac{1}{4}}) +      18 i \cot^{-1}(  \frac{-2 + 2^{\frac{1}{4}} \sqrt{2 + \sqrt{2}}}{2^{\frac{1}{4}} \sqrt{2 - \sqrt{2}}} )     \ln(1 + (-1)^{\frac{1}{8}} 2^{\frac{1}{4}})     \\    
  &  & -   18 i \cot^{-1}   (    \frac{2 + 2^{\frac{1}{4}} \sqrt{2 + \sqrt{2}}}{   2^{\frac{1}{4}} \sqrt{2 - \sqrt{2}}}   )        \ln(1 + (-1)^{\frac{1}{8}} 2^{\frac{1}{4}})    +   9 \ln   ^{2}       (   \frac{1 + (-1)^{\frac{1}{8}} 2^{\frac{1}{4}}}{   1 - (-1)^{\frac{1}{8}} 2^{\frac{1}{4}}}   )       \\   
  &  &    - (\frac{522}{35} + \frac{124 i}{5}) (-1)^{\frac{7}{8}} 2^{\frac{1}{4}} \ln(-1 + (-1)^{\frac{7}{8}} 2^{\frac{1}{4}}) +      18 i \cot^{-1}   (   \frac{2 + 2^{\frac{1}{4}} \sqrt{2 + \sqrt{2}}}{ 2^{\frac{1}{4}} \sqrt{2 - \sqrt{2}}}     )     \ln(-1 + (-1)^{\frac{7}{8}} 2^{\frac{1}{4}})        \\  
  &  &  + (\frac{522}{35}        +   \frac{124i}{5} ) (-1)^{\frac{7}{8}} 2^{\frac{1}{4}} \ln(1 + (-1)^{\frac{7}{8}} 2^{\frac{1}{4}}) -      18 \pi i  \ln(1 + (-1)^{\frac{7}{8}} 2^{\frac{1}{4}})      +     18 i \cot^{-1}    (  \frac{-2 + 2^{\frac{1}{4}} \sqrt{2 + \sqrt{2}}}{2^{\frac{1}{4}} \sqrt{2 - \sqrt{2}}}    )    \\      
  &  &  \ln(1 + (-1)^{\frac{7}{8}} 2^{\frac{1}{4}}) -     18 i \cot^{-1}    (  \frac{2 + 2^{\frac{1}{4}} \sqrt{2 + \sqrt{2}}}{  2^{\frac{1}{4}} \sqrt{2 - \sqrt{2}}}  ) \ln[1 + (-1)^{\frac{7}{8}} 2^{\frac{1}{4}}]      \\    
  &  &   +  18 \ln(-1 + (-1)^{\frac{7}{8}} 2^{\frac{1}{4}}) \ln(1 + (-1)^{\frac{7}{8}} 2^{\frac{1}{4}})      \\    
  &  &   -  9 \ln(1 + (-1)^{\frac{7}{8}} 2^{\frac{1}{4}}) \ln^{2 }( \operatorname{Im}(\sqrt{1 - i})+ (-1 + \operatorname{Re}(\sqrt{1 - i}))^2)      \\    
  &  &    -  9 \ln(-1 + (-1)^{\frac{7}{8}} 2^{\frac{1}{4}}) \ln^{2}(    \operatorname{Im}(\sqrt{1 - i}) + (1 + \operatorname{Re}(\sqrt{1 - i}))^2)      \\    
  &  &     +  9 \ln(1 + (-1)^{\frac{7}{8}} 2^{\frac{1}{4}}) \ln( \operatorname{Im}(\sqrt{1 - i})^2 + (1 + \operatorname{Re}(\sqrt{1 - i}))^{2})      \\    
  &  &     -  9 \ln(-1 + (-1)^{\frac{1}{8}} 2^{\frac{1}{4}}) \ln(  \operatorname{Im}(\sqrt{1 + i})^2 + (-1 + \operatorname{Re}(\sqrt{1 + i}))^2)      \\    
  &  &    +  9 \ln(1 + (-1)^{\frac{1}{8}} 2^{\frac{1}{4}}) \ln(  \operatorname{Im}(\sqrt{1 + i})^2 + (-1 + \operatorname{Re}(\sqrt{1 + i}))^2)      \\    
  &  &    +   9 \ln(-1 + (-1)^{\frac{1}{8}} 2^{\frac{1}{4}}) \ln(  \operatorname{Im}(\sqrt{1 + i})^2 + (1 + \operatorname{Re}(\sqrt{1 + i}))^2)      \\    
  &  &    -  9 \ln(1 + (-1)^{\frac{1}{8}} 2^{\frac{1}{4}}) \ln(  \operatorname{Im}(\sqrt{1 + i})^2 + (1 + \operatorname{Re}(\sqrt{1 + i}))^2)   
\end{eqnarray*}
\end{example}%

\begin{proof}

  Use the HJ-algorithm.

\begin{equation*}
\sum_{n=1}^{\infty}      \frac{(1 - x^2)^{n} x^{4 n}}{2^{n} (n + 2) }  = \frac{-4 x^{4} + 4 x^{6} - x^{8} + 2 x^{10} - x^{12} -  8 \ln(\frac{1}{2} (2 - x^{4} + x^{6}))}{2 x^{8} (-1 + x^2)^2}
\end{equation*}%

\end{proof}

\begin{theorem}%
\begin{equation*}
_{3} F _{2}   (\frac{3}{2},2,2;\frac{7}{4},\frac{9}{4};\frac{1}{3}) =  \frac{45}{128} (6 - \sqrt{3 (1 + \sqrt{3})} \coth^{-1}(\sqrt{1 + \sqrt{3}} ) +  \sqrt{3 (-1 + \sqrt{3})} \csc^{-1}(3^{\frac{1}{4}}))
\end{equation*}
\end{theorem}%

\begin{proof}

  Use the HJ-algorithm.

\begin{equation*}
\sum_{n=1}^{\infty} \frac{(1 - x^2)^{2n} n}      {3^n}    =  \frac{3 (-1 + x^2)^2}{(-2 - 2 x^{2} + x^4)^2} 
\end{equation*}

\end{proof}

\begin{theorem}%
\begin{equation*}
_{5} F _{4}     (\frac{3}{2},2,2,2,2;1,\frac{7}{4},\frac{9}{4},3;\frac{1}{4}) =  \frac{5}{24} (9 (-6 + \pi) \pi +  4 (6 + 14 \sqrt{3} \cosh^{-1}(2) + 9 \cosh^{-1}(2)^{2} -   72 \coth^{-1}(\sqrt{3})^2))
\end{equation*}
\end{theorem}%

\begin{proof}

  Use the HJ-algorithm.

\begin{eqnarray*}
\lefteqn{\sum_{n=1}^\infty \frac{(1 - x^2)^{2 n} n^2}{4^{n} (n + 1)} = }   \\
  &  &   -\frac{1}{(3 - x^{2} - 3 x^{4} + x^{6})^2}  4 (2 - 8 x^{4} + 8 x^{6} - 2 x^{8} + 9 \ln(\frac{1}{4} (3 + 2 x^{2} - x^{4})) +  12 x^{2} \ln(\frac{1}{4} (3 + 2 x^{2} - x^{4}))   \\
  &  &  -  2 x^{4} \ln(\frac{1}{4} (3 + 2 x^{2} - x^{4})) -   4 x^{6} \ln(\frac{1}{4} (3 + 2 x^{2} - x^{4}))    + x^{8} \ln(\frac{1}{4} (3 + 2 x^{2} - x^{4})))
\end{eqnarray*}

\end{proof}

\begin{theorem}%
\begin{eqnarray*}
\lefteqn{_{3} F _{2}   (\frac{3}{2},2,2;\frac{5}{3},\frac{13}{6};\frac{1}{4}     )    =   }    \\
  &  &    \frac{14}{2187}     (   216 +   \sqrt{3}    (18+53^{\frac{1}{3}})  \pi   -   3 ^ {\frac{5}{6}}   30 \tan^{-1}   (\frac{2+3^{\frac{1}{3}}}{3^\frac{5}{6}})      + 54 \ln (2)      \\ 
  &  &   + 3 ^ \frac{1}{3}  30 \ln(3-3^\frac{2}{3})   -  3 ^ \frac{1}{3}   15 \ln (3(3+3^{\frac{1}{3}}+3^{\frac{2}{3}}))           )
\end{eqnarray*}
\end{theorem}%

\begin{proof}

  Use the HJ-algorithm.

\begin{equation*}
\sum_{n=1}^{\infty} \frac{(1 - x^3)^{2 n} n}{4^n} = \frac{4 (-1 + x^3)^2}{(-3 - 2 x^{3} + x^6)^2}
\end{equation*}%

\end{proof}

\begin{theorem}%
\begin{equation*}
_{4} F _{3} (\frac{3}{2}, 2, 2, 2; 1, \frac{13}{8}, \frac{17}{8}; \frac{1}{4})  = \frac{5}{1024}(88 + 63 \sqrt{2} \pi + 3^{\frac{1}{4}} 4 \cot^{-1}(3^{\frac{1}{4}}) + 
 126 \sqrt{2} \coth^{-1}(\sqrt{2}) +  3^{\frac{1}{4}}4 \coth^{-1}(3^{\frac{1}{4}}))
\end{equation*}
\end{theorem}%

\begin{proof}

  Use the HJ-algorithm.

\begin{equation*}
\sum_{n=1}^{\infty} \frac{(1 - x^4)^{2 n} n^2}{4^n} = -\frac{4 (-1 + x^4)^2 (5 - 2 x^{4} + x^8)}{(-3 - 2 x^4 + x^8)^3}
\end{equation*}

\end{proof}

\begin{example}%
\begin{equation*}
_{3}  F _{2}   (-\frac{3}{2},\frac{1}{4},\frac{3}{4};\frac{3}{2},2;1) =   - \frac{9}{40 \sqrt{2} \pi} +   \frac{55 \tanh^{-1}(\frac{1}{\sqrt{2}})}{16 \pi}  
\end{equation*}
\end{example}%

\begin{example}%
\begin{eqnarray*}
\lefteqn{  _{4}  F _{3}   (1,1,1,\frac{3}{2};\frac{7}{4},2,\frac{9}{4};1) =   }      \\  
  &  &      (-8 + \frac{7 \pi^{2}}{6} + \frac{1}{2} \ln^{2}(-i - 2^{\frac{1}{4}}) +         \ln^{2}(i - 2^{\frac{1}{4}}) + 2 \sqrt{2} \ln(1 - i 2^{\frac{1}{4}}) -       \ln(2) \ln(1 - i 2^{\frac{1}{4}})   \\ 
  &  &     + 2 \sqrt{2} \ln(1 + i 2^{\frac{1}{4}}) -      \ln(2) \ln(1 + i 2^{\frac{1}{4}}) - 2 \sqrt{2} \ln(-1 + 2^{\frac{1}{4}}) +     \ln(2) \ln(-1 + 2^{\frac{1}{4}})    \\ 
  &  &     - \ln(-i - 2^{\frac{1}{4}}) \ln(-1 + 2^{\frac{1}{4}}) -      \ln(i - 2^{\frac{1}{4}}) \ln(-1 + 2^{\frac{1}{4}}) +     \ln(1 - i 2^{\frac{1}{4}}) \ln(-1 + 2^{\frac{1}{4}})    \\ 
  &  &     +        \ln(1 + i 2^{\frac{1}{4}}) \ln(-1 + 2^{\frac{1}{4}}) - \frac{1}{2}  \ln^{2}(-1 + 2^{\frac{1}{4}}) +      \ln(1 + i 2^{\frac{1}{4}}) \ln(-   \frac{i}{-i + 2^{\frac{1}{4}}}  )        \\ 
  &  &     +        \ln(1 - i 2^{\frac{1}{4}})    \ln   (     -  \frac{2 i}{-i + 2^{\frac{1}{4}}}      ) -    \ln(-1 + 2^{\frac{1}{4}}) \ln   (    \frac{1 - i}{-i + 2^{\frac{1}{4}}}   )       +        \frac{1}{2} \ln ^{2}  (   -    \frac{1}{i + 2^{\frac{1}{4}}}     )    \\  
  &  &   +      \ln(1 - i 2^{\frac{1}{4}}   )    \ln(\frac{i}{i + 2^{\frac{1}{4}}}   )   +   \ln(1 + i 2^{\frac{1}{4}}) \ln   (  \frac{2 i}{i + 2^{\frac{1}{4}}}   )         +      i \pi (\ln   (   1 - i 2^{\frac{1}{4}}   )          \\ 
  &  &      -    \ln(1 + i 2^{\frac{1}{4}}   )      +          \ln    (    -     \frac{2 i}{-i + 2^{\frac{1}{4}}}             )             - \ln    ( \frac{1 - i}{-i + 2^{\frac{1}{4}}}           ) +     \ln   (  \frac{i}{i + 2^{\frac{1}{4}}}   )   -    \ln   (   \frac{1 + i}{i + 2^{\frac{1}{4}}}      )   )        \\
  &  &      -       \ln   (   -1 + 2^{\frac{1}{4}}   )    \ln   (    \frac{1 + i}{i + 2^{\frac{1}{4}}}        ) -       2 \sqrt{2}   \ln   (   1 + 2^{\frac{1}{4}}   ) + \ln(2)  \ln   (    1 + 2^{\frac{1}{4}}) - \ln   (    i - 2^{\frac{1}{4}}   )    \ln  (1 + 2^{\frac{1}{4}})          \\  
  &  &     +       \ln (1 - i 2^{\frac{1}{4}}) \ln(1 + 2^{\frac{1}{4}})        +        \ln(1 + i 2^{\frac{1}{4}}) \ln(1 + 2^{\frac{1}{4}}) -       \ln(-1 + 2^{\frac{1}{4}}) \ln(1 + 2^{\frac{1}{4}})      \\  
  &  &     -        \ln  (   -  \frac{1 + i}{-i + 2^{\frac{1}{4}}}   ) \ln(1 + 2^{\frac{1}{4}}) +      \ln   (    -  \frac{1}{i + 2^{\frac{1}{4}}}    ) \ln(1 + 2^{\frac{1}{4}})            \\  
  &  &   -       \ln( -  \frac{1 - i}{i + 2^{\frac{1}{4}}}   ) \ln(1 + 2^{\frac{1}{4}}) -       \frac{1}{2} \ln^{2}(1 + 2^{\frac{1}{4}})          +         \frac{1}{3}  (\pi^{2} - 6 \pi \cot^{-1} (2^{\frac{1}{4}}) +       6 \cot^{-1}(2^{\frac{1}{4}})^2)) \frac{15}{4}
\end{eqnarray*}
\end{example}%

  Given that $\sum_{n=1}^{\infty}      \frac{((1 - x^4)^{2 n}) x}{n^2} =   x \operatorname{Li}_{2} ((-1 + x^4)^2)$, integrate both sides of this
equation.

  What are alternative series representations of series of the form  $\sum_{n=2}^{\infty}   \frac{(\zeta (n) - 1) (p)_{an}}   {(q)_{bn}} $    ?

\begin{theorem}%
\begin{eqnarray*}
\lefteqn{\sum_{n=2}^{\infty}    \frac{(\zeta (n) - 1)(2 n)!}   {(\frac{5}{4})_{2n}}   = }    \\   
  &  &   \sum_{n=2}^{\infty}           -(\frac{1}{8} (8 -  \frac{  \sqrt{n} \ln(1 - (1 - \sqrt{n})^{\frac{1}{4}})  }{    (1 -  \sqrt{n})^{\frac{3}{4}}   }   -   \frac{i \sqrt{n} \ln(     1 - i (1 - \sqrt{n})^{\frac{1}{4}})}{ (1 - \sqrt{n})^{\frac{3}{4}}}         +     \frac{i \sqrt{n} \ln(1 + i (1 - \sqrt{n})^{\frac{1}{4}})}{(1 -    \sqrt{n})^{\frac{3}{4}}}      \\  
  &  &       +   \frac{\sqrt{n} \ln(       1 + (1 - \sqrt{n})^{\frac{1}{4}})}{(1 - \sqrt{n})^{\frac{3}{4}} }           +        \frac{\sqrt{n} \ln(1 - (1 +   \sqrt{n})^{\frac{1}{4}})}{(1 + \sqrt{n})^{\frac{3}{4}}}       +       \frac{i \sqrt{n} \ln(1 - i (1 + \sqrt{n})^{\frac{1}{4}})}{(1 +    \sqrt{n})^{\frac{3}{4}}}      \\  
  &  &     -     \frac{i \sqrt{n} \ln(     1 + i (1 + \sqrt{n})^{\frac{1}{4}})}{(1 + \sqrt{n})^{\frac{3}{4}}}     -  \frac{\sqrt{ n} \ln(1 + (1 + \sqrt{n})^{\frac{1}{4}})}{(1 + \sqrt{n})^{\frac{3}{4}}}      \\  
  &  &     - \frac{\sqrt{ n}}{32 (1 - n)^{\frac{3}{4}}}  (-4 (1 + \sqrt{n})^{\frac{3}{4}} \ln(-(1 - \sqrt{n})^{\frac{1}{4}})   4 i (1 + \sqrt{n})^{\frac{3}{4}} \ln(-i (1 - \sqrt{n})^{\frac{1}{4}})     \\ 
  &  &    +      4 i (1 + \sqrt{n})^{\frac{3}{4}} \ln(   i (1 - \sqrt{n})^{\frac{1}{4}}) + (1 + \sqrt{n})^{\frac{3}{4}} \ln(      1 - \sqrt{n})  +       4 (1 - \sqrt{n})^{\frac{3}{4}} \ln(-(1 + \sqrt{n})^{\frac{1}{4}})    \\  
  &  &  +     4 i (1 - \sqrt{n})^{\frac{3}{4}} \ln(-i (1 + \sqrt{n})^{\frac{1}{4}}) -      4 i (1 - \sqrt{n})^{\frac{3}{4}} \ln(  i (1 + \sqrt{n})^{\frac{1}{4}})      \\ 
  &  &   - (1 - \sqrt{n})^{\frac{3}{4}}   \ln(  1 + \sqrt{n})) - \frac{32}{45 n})  
\end{eqnarray*}
\end{theorem}%

\begin{proof}

  Consider:

\begin{equation*}
\label{eq:expan1}
\sum_{n=2}^{\infty} \frac{(1-x^4)^{2n}}{c^n} =   \frac{(-1 + x^4)^4}{c (-1 + c + 2 x^{4} - x^8)}
\end{equation*}

  Integrate both sides of the above equation; integrating the right-hand side proves to be rather challenging.

\end{proof}

  A question which immediately comes to mind is: ``Can the right-hand side of {\bf Theorem 17} be simplified at all?" We presently leave open
to discussion how such $\zeta  (x)$ expansions can be manipulated and simplified. We should note that such $\zeta  (x) $ expansions may be expressed in integral
 form by evaluating series of the form $\sum_{n=2}^{\infty} \zeta  (x) f(x)^{n}$. Such definite integrals are often rather striking.

  The following relationship is related to {\bf Theorem 17}.

\begin{theorem}%
\begin{eqnarray*}
\lefteqn{ _{3} F _{2}   ( 1,3,\frac{5}{2}   ; \frac{21}{8}  ,  \frac{25}{8}     ;      \frac{1}{n})  \frac{8192}{n^{2}13260}       =   }    \\ 
  &  &     -(\frac{1}{8} (8 -  \frac{  \sqrt{n} \ln(1 - (1 - \sqrt{n})^{\frac{1}{4}})  }{    (1 -  \sqrt{n})^{\frac{3}{4}}   }   -   \frac{i \sqrt{n} \ln(     1 - i (1 - \sqrt{n})^{\frac{1}{4}})}{ (1 - \sqrt{n})^{\frac{3}{4}}}         +     \frac{i \sqrt{n} \ln(1 + i (1 - \sqrt{n})^{\frac{1}{4}})}{(1 -    \sqrt{n})^{\frac{3}{4}}}      \\    
  &  &       +   \frac{\sqrt{n} \ln(       1 + (1 - \sqrt{n})^{\frac{1}{4}})}{(1 - \sqrt{n})^{\frac{3}{4}} }           +        \frac{\sqrt{n} \ln(1 - (1 +   \sqrt{n})^{\frac{1}{4}})}{(1 + \sqrt{n})^{\frac{3}{4}}}       +       \frac{i \sqrt{n} \ln(1 - i (1 + \sqrt{n})^{\frac{1}{4}})}{(1 +    \sqrt{n})^{\frac{3}{4}}}      \\  
  &  &     -     \frac{i \sqrt{n} \ln(     1 + i (1 + \sqrt{n})^{\frac{1}{4}})}{(1 + \sqrt{n})^{\frac{3}{4}}}     -  \frac{\sqrt{ n} \ln(1 + (1 + \sqrt{n})^{\frac{1}{4}})}{(1 + \sqrt{n})^{\frac{3}{4}}}      \\  
  &  &     - \frac{\sqrt{ n}}{32 (1 - n)^{\frac{3}{4}}}  (-4 (1 + \sqrt{n})^{\frac{3}{4}} \ln(-(1 - \sqrt{n})^{\frac{1}{4}})   4 i (1 + \sqrt{n})^{\frac{3}{4}} \ln(-i (1 - \sqrt{n})^{\frac{1}{4}})     \\ 
  &  &    +      4 i (1 + \sqrt{n})^{\frac{3}{4}} \ln(   i (1 - \sqrt{n})^{\frac{1}{4}}) + (1 + \sqrt{n})^{\frac{3}{4}} \ln(      1 - \sqrt{n})  +       4 (1 - \sqrt{n})^{\frac{3}{4}} \ln(-(1 + \sqrt{n})^{\frac{1}{4}})    \\  
  &  &  +     4 i (1 - \sqrt{n})^{\frac{3}{4}} \ln(-i (1 + \sqrt{n})^{\frac{1}{4}}) -      4 i (1 - \sqrt{n})^{\frac{3}{4}} \ln(  i (1 + \sqrt{n})^{\frac{1}{4}})      \\ 
  &  &   - (1 - \sqrt{n})^{\frac{3}{4}}   \ln(  1 + \sqrt{n})) - \frac{32}{45 n})  
\end{eqnarray*}
\end{theorem}%

\begin{proof}

  Integrate (\ref{eq:expan1}).

\end{proof}

\begin{theorem}%
\begin{equation*}
_{4}   F   _{3}     (1,1,\frac{5}{2},3;2,\frac{21}{8},\frac{25}{8};1)   =      \frac{221}{276480}      (12584 - 2^{\frac{1}{4}} 1305 \pi -2^{\frac{1}{4}} 2610 \coth^{-1}(2^{\frac{1}{4}}) +    2^{\frac{1}{4}} 2610  \tan^{-1}(2^{\frac{1}{4}}))
\end{equation*}
\end{theorem}%

\begin{proof}

  Integrate both sides of {\bf Theorem 18}.

\end{proof}

\begin{theorem}%
\begin{equation}
_{3} F _{2}     ( 1,  \frac{3}{2}   ,   3  ;    \frac{21}{8}    ,     \frac{25}{8}  ;  1    )    =         \frac{221 }{4096}   (   8 + 2^{\frac{1}{4}}  27  \pi +  2^{\frac{1}{4}}  54 \coth^{-1}(2^{\frac{1}{4}}) -   2^{\frac{1}{4}}  54  \tan^{-1}(2^{\frac{1}{4}})   ) 
\end{equation}
\end{theorem}%

\begin{proof}

   Replace $x$ by $x^2$ of and integrate.

\end{proof}

  Clearly, integrating ``generalized" power series may lead to outlandish results. Let us continue in a similar vein.

\begin{theorem}%
\begin{eqnarray*}
\lefteqn{          _{5} F _{4}      (1,\frac{9}{4},\frac{5}{2},\frac{11}{4},3;\frac{37}{16},\frac{41}{16},\frac{45}{16},\frac{49}{16} ;\frac{1}{c})     \frac{8388608  }{15862275 c^2}  =    }      \\    
  &  &      \frac{1}{c}       (       -   \frac{2048}{3315}   - c    +          \frac{ c^{\frac{5}{4}}     \ln     (     1 - (1 -     c^{\frac{1}{4}}    )^{\frac{1}{4}}   )    }{16    (   1 - c^{\frac{1}{4}})^{\frac{3}{4}}}            +    \frac{     i c^{\frac{5}{4}}        \ln     (      1 - i (1 - c^{\frac{1}{4}})^{\frac{1}{4}}     ) }{    16     (1    -    c^{\frac{1}{4}})^{\frac{3}{4}} }            -     \frac{  i c^{\frac{5}{4}} \ln(1 + i (1 - c^{\frac{1}{4}})^{\frac{1}{4}}) }{ 16 (1 - c^{\frac{1}{4}})^{\frac{3}{4}}}         \\  
  &  &       -      \frac{   c^{\frac{5}{4}} \ln(1 + (1 - c^{\frac{1}{4}})^{\frac{1}{4}}) }{  16 (1 - c^{\frac{1}{4}})^{\frac{3}{4}}}               +  \frac{     i c^{\frac{5}{4}} \ln(1 - (1 - i c^{\frac{1}{4}})^{\frac{1}{4}}) }{  16 (1 - i c^{\frac{1}{4}})^{\frac{3}{4}}}        - \frac{ c^{\frac{5}{4}} \ln(1 - i (1 - i c^{\frac{1}{4}})^{\frac{1}{4}}) }{ 16 (1 - i c^{\frac{1}{4}})^{\frac{3}{4}}}           +  \frac{  c^{\frac{5}{4}} \ln(1 + i (1 - i c^{\frac{1}{4}})^{\frac{1}{4}}) }{ 16 (1 - i c^{\frac{1}{4}})^{\frac{3}{4}}}          \\
  &  &       -      \frac{  i c^{\frac{5}{4}} \ln(1 + (1 - i c^{\frac{1}{4}})^{\frac{1}{4}}) }{  16 (1 - i c^{\frac{1}{4}})^{\frac{3}{4}}}            -  \frac{  i c^{\frac{5}{4}} \ln(1 - (1 + i c^{\frac{1}{4}})^{\frac{1}{4}}) }{ 16 (1 + i c^{\frac{1}{4}})^{\frac{3}{4}}}    +   \frac{  c^{\frac{5}{4}} \ln(1 - i (1 + i c^{\frac{1}{4}})^{\frac{1}{4}}) }{  16 (1 + i c^{\frac{1}{4}})^{\frac{3}{4}}}          -  \frac{  c^{\frac{5}{4}} \ln(1 + i (1 + i c^{\frac{1}{4}})^{\frac{1}{4}}) }{  16 (1 + i c^{\frac{1}{4}})^{\frac{3}{4}}}           \\
  &  &      +      \frac{ i c^{\frac{5}{4}} \ln(1 + (1 + i c^{\frac{1}{4}})^{\frac{1}{4}})}{ 16 (1 + i c^{\frac{1}{4}})^{\frac{3}{4}}}       - \frac{  c^{\frac{5}{4}} \ln(1 - (1 + c^{\frac{1}{4}})^{\frac{1}{4}})  }{16 (1 + c^{\frac{1}{4}})^{\frac{3}{4}}}       -  \frac{  i c^{\frac{5}{4}} \ln(1 - i (1 + c^{\frac{1}{4}})^{\frac{1}{4}})}{ 16 (1 + c^{\frac{1}{4}})^{\frac{3}{4}}}           +   \frac{   i c^{\frac{5}{4}} \ln(1 + i (1 + c^{\frac{1}{4}})^{\frac{1}{4}})}{ 16 (1 + c^{\frac{1}{4}})^{\frac{3}{4}}}        \\ 
  &  &       +  \frac{c^{\frac{5}{4}} \ln(1 + (1 + c^{\frac{1}{4}})^{\frac{1}{4}}) }{ 16 (1 + c^{\frac{1}{4}})^{\frac{3}{4}}}            -  \frac{1}{64}       c^{\frac{5}{4}}        \frac{(4 \ln(-(1 - c^{\frac{1}{4}})^{\frac{1}{4}})}{(1 - c^{\frac{1}{4}})^{\frac{3}{4}} }     +   \frac{ 4 i \ln(-i (1 - c^{\frac{1}{4}})^{\frac{1}{4}})}{(1 - c^{\frac{1}{4}})^{\frac{3}{4}}}           -  \frac{ 4 i \ln(i (1 - c^{\frac{1}{4}})^{\frac{1}{4}})}{(1 - c^{\frac{1}{4}})^{\frac{3}{4}} }                  \\  
  &  &          -            \frac{\ln(1 - c^{\frac{1}{4}}) }{  (1 - c^{\frac{1}{4}})^{\frac{3}{4}} }                +       \frac{ 4 i \ln(-(1 - i c^{\frac{1}{4}})^{\frac{1}{4}}) }{ (1 - i c^{\frac{1}{4}})^{\frac{3}{4}} }         -          \frac{   4 \ln(-i (1 - i c^{\frac{1}{4}})^{\frac{1}{4}})}{  (1 - i c^{\frac{1}{4}})^{\frac{3}{4}}}            +      \frac{      4 \ln(i (1 - i c^{\frac{1}{4}})^{\frac{1}{4}})  }{  (1 - i c^{\frac{1}{4}})^{\frac{3}{4}}}                  -    \frac{   i \ln(1 - i c^{\frac{1}{4}}) }{  (1 - i c^{\frac{1}{4}})^{\frac{3}{4}} }           \\   
  &  &                -     \frac{     4 i \ln(-(1 + i c^{\frac{1}{4}})^{\frac{1}{4}})}{(1 + i c^{\frac{1}{4}})^{\frac{3}{4}}}          +      \frac{    4 \ln(-i (1 + i c^{\frac{1}{4}})^{\frac{1}{4}})}{  (1 + i c^{\frac{1}{4}})^{\frac{3}{4}}}      -   \frac{     4 \ln(i (1 + i c^{\frac{1}{4}})^{\frac{1}{4}})}{ (1 + i c^{\frac{1}{4}})^{\frac{3}{4}}}             +      \frac{    i \ln(1 + i c^{\frac{1}{4}})}{  (1 + i c^{\frac{1}{4}})^{\frac{3}{4}} }                -     \frac{     4 \ln(-(1 + c^{\frac{1}{4}})^{\frac{1}{4}}) }{    (1 + c^{\frac{1}{4}})^{\frac{3}{4}}}             \\ 
  &  &              -    \frac{     4 i \ln(-i (1 + c^{\frac{1}{4}})^{\frac{1}{4}}) }{  (1 + c^{\frac{1}{4}})^{\frac{3}{4}} }          +     \frac{    4 i \ln(i (1 + c^{\frac{1}{4}})^{\frac{1}{4}}) }{  (1 + c^{\frac{1}{4}})^{\frac{3}{4}}}     +    \frac{  \ln(1 + c^{\frac{1}{4}}) }{  (1 + c^{\frac{1}{4}})^{\frac{3}{4}})}      )            
\end{eqnarray*}  
\end{theorem}%

\begin{proof}

  Consider the appealing sum $\sum_{n=2}^{\infty}       \frac{(1-x^4)^{4n}}{c^n}        =       \frac{(-1 + x^4)^8}{c (-1 + c + 4 x^{4} - 6 x^{8} + 4 x^{12} - x^16)} $.
 Integrate both sides of this equation.

\end{proof}

  Consider integrating both sides of the above theorem.

\begin{eqnarray*}
\lefteqn{     _{6}   F _{5}    (1,1,\frac{9}{4},\frac{5}{2},\frac{11}{4},3;2,\frac{37}{16},\frac{41}{16},\frac{45}{16},\frac{49}{16};1)    =  }              \\   
  &  &       (     \frac{332799152   }{ 32967675 }      +   \frac{      (        (272 - 1923 i) (1 - i)^{\frac{1}{4}} + (272 + 1923 i) (1 + i)^{\frac{1}{4}} - 2^{\frac{1}{4}}  28976 ) \pi  }{53040 }           \\   
  &  &   +    \frac{(     (1923 + 272 i) (1 - i)^{\frac{1}{4}} + (1923 - 272 i) (1 + i)^{\frac{1}{4}} ) \ln 2 }{  26520 }     - (\frac{641}{1105} +  \frac{16 i}{195} ) (1 - i)^{\frac{1}{4}}         \\   
  &  &         \ln(-(1 - i)^{\frac{1}{4}}) + (\frac{16}{195} - \frac{641 i}{1105})    (1 - i)^{\frac{1}{4}}             \ln(-i (1 - i)^{\frac{1}{4}}) - (\frac{16}{195} - \frac{641 i}{1105})   (1 - i)^{\frac{1}{4}}         \\   
  &  &         \ln(i (1 - i)^{\frac{1}{4}}) - (\frac{641}{1105} -  \frac{16 i}{195}  )   (1 + i)^{\frac{1}{4}}           \ln(-(1 + i)^{\frac{1}{4}}) - (\frac{16}{195} + \frac{641 i}{1105})    (1 + i)^{\frac{1}{4}}         \\   
  &  &         \ln(-i (1 + i)^{\frac{1}{4}}) + (\frac{16}{195} + \frac{641 i}{1105}) (1 + i)^{\frac{1}{4}}              \ln(i (1 + i)^{\frac{1}{4}}) + (\frac{641}{1105} +  \frac{16 i}{195} ) (1 - i)^{\frac{1}{4}}         \\   
  &  &         \ln(1 - (1 - i)^{\frac{1}{4}}) - (\frac{16}{195} - \frac{641 i}{1105}) (1 - i)^{\frac{1}{4}}             \ln(1 - i (1 - i)^{\frac{1}{4}}) + (\frac{16}{195}     - \frac{641 i}{1105}) (1 - i)^{\frac{1}{4}}         \\   
  &  &         \ln(1 + i (1 - i)^{\frac{1}{4}}) - (\frac{641}{1105}   +  \frac{16 i}{195} ) (1 - i)^{\frac{1}{4}}          \ln(1 + (1 - i)^{\frac{1}{4}}) + (\frac{641}{1105} -  \frac{16 i}{195} ) (1 + i)^{\frac{1}{4}}         \\   
  &  &         \ln(1 - (1 + i)^{\frac{1}{4}}) + (\frac{16}{195} + \frac{641 i}{1105}) (1 + i)^{\frac{1}{4}}              \ln(1 - i (1 + i)^{\frac{1}{4}}) - (\frac{16}{195} + \frac{641 i}{1105}) (1 + i)^{\frac{1}{4}}         \\   
  &  &         \ln(1 + i (1 + i)^{\frac{1}{4}}) - (\frac{641}{1105} -  \frac{16 i}{195} ) (1 + i)^{\frac{1}{4}}            \ln(1 + (1 + i)^{\frac{1}{4}}) +  \frac{  1811 i 2^{\frac{1}{4}} \ln(1 - i 2^{\frac{1}{4}})}{  3315 }            \\   
  &  &    -      \frac{ 1811 i 2^{\frac{1}{4}} \ln(1 + i 2^{\frac{1}{4}}) }{ 3315}   +  \frac{ 1811 2^{\frac{1}{4}} \ln(-1 + 2^{\frac{1}{4}}) }{  3315 }          -    \frac{ 2^{\frac{1}{4}}   1811\ln(1 + 2^{\frac{1}{4}})}{3315}  )  \frac{15862275}{8388608}       
\end{eqnarray*}

  Have hypergeometric series such as the above one been evaluated, published, or widely recognized?

\begin{theorem}%
\begin{eqnarray*}
\lefteqn{    _{3} F _{2}         (\frac{3}{2},2,2;\frac{13}{8},\frac{17}{8};\frac{1}{c})            =          }           \\       
  &  &    1-\frac{45}{32}    (   -   (   \frac{1}{    2880 (1 - c)^{\frac{7}{4}}  }     )      (-2048 (1 - c)^{\frac{3}{4}} + 2048 (1 - c)^{\frac{3}{4}} c -            360 (1 - c)^{\frac{3}{4}} c^{2}      +       45 (1 + \sqrt{c})^{\frac{3}{4}} c^{\frac{3}{2}}      \\    
  &  &     (-4 - 3 \sqrt{c} + c) \ln(1 - (1 - \sqrt{c})^{\frac{1}{4}}) +      45 i (1 + \sqrt{c})^{\frac{3}{4}} c^{\frac{3}{2}}           (-4 - 3 \sqrt{c} + c)       \ln(1 - i (1 - \sqrt{c})^{\frac{1}{4}})            \\  
  &  &          +               180 i (1 + \sqrt{c})^{\frac{3}{4}}     c^{\frac{3}{2}}      \ln(1 + i (1 - \sqrt{c})^{\frac{1}{4}})     +    135 i (1 + \sqrt{c})^{\frac{3}{4}} c^2 \ln(1 + i (1 - \sqrt{c})^{\frac{1}{4}})         -        45 i (1 + \sqrt{c})^{\frac{3}{4}} c^{\frac{5}{2}}     \\  
  &  &             \ln(1 + i (1 - \sqrt{c})^{\frac{1}{4}}) +     180 (1 + \sqrt{c})^{\frac{3}{4}}     c^{\frac{3}{2}} \ln(1 + (1 - \sqrt{c})^{\frac{1}{4}}) +             135 (1 + \sqrt{c})^{\frac{3}{4}} c^{2} \ln(1 + (1 - \sqrt{c})^{\frac{1}{4}})     \\       
  &  &      -       45 (1 + \sqrt{c})^{\frac{3}{4}} c^{\frac{5}{2}} \ln(1 + (1 - \sqrt{c})^{\frac{1}{4}}) +             180 (1 - \sqrt{c})^{\frac{3}{4}} c^{\frac{3}{2}} \ln(1 - (1 + \sqrt{c})^{\frac{1}{4}}) -            135 (1 - \sqrt{c})^{\frac{3}{4}} c^{2}     \\  
  &  &     \ln(1 - (1 + \sqrt{c})^{\frac{1}{4}})         -     45 (1 - \sqrt{c})^{\frac{3}{4}} c^{\frac{5}{2}} \ln(1 - (1 + \sqrt{c})^{\frac{1}{4}}) +             180 i (1 - \sqrt{c})^{\frac{3}{4}} c^{\frac{3}{2}}              \ln(1 - i (1 + \sqrt{c})^{\frac{1}{4}}) -     \\       
  &  &          135 i (1 - \sqrt{c})^{\frac{3}{4}} c^{2} \ln(1 - i (1 + \sqrt{c})^{\frac{1}{4}}) -      45 i (1 - \sqrt{c})^{\frac{3}{4}} c^{\frac{5}{2}}           \ln(1 - i (1 + \sqrt{c})^{\frac{1}{4}})            -                     180 i (1 - \sqrt{c})^{\frac{3}{4}} c^{\frac{3}{2}}       \\ 
  &  &           \ln(1 + i (1 + \sqrt{c})^{\frac{1}{4}})           +                     135 i (1 - \sqrt{c})^{\frac{3}{4}} c^2 \ln(1 + i (1 + \sqrt{c})^{\frac{1}{4}}) +              45 i (1 - \sqrt{c})^{\frac{3}{4}} c^{\frac{5}{2}}                     \ln(1 + i (1 + \sqrt{c})^{\frac{1}{4}})          \\  
  &  &       -           180 (1 - \sqrt{c})^{\frac{3}{4}} c^{\frac{3}{2}} \ln(1 + (1 + \sqrt{c})^{\frac{1}{4}})       +      135 (1 - \sqrt{c})^{\frac{3}{4}} c^2 \ln(1 + (1 + \sqrt{c})^{\frac{1}{4}})      +       45 (1 - \sqrt{c})^{\frac{3}{4}} c^{\frac{5}{2}}           \\  
  &  &      \ln(1 + (1 + \sqrt{c})^{\frac{1}{4}})) +          (   \frac{1}{   256      (1 - c)^{\frac{7}{4}} }      )    c^{\frac{3}{2}} (4 (1 + \sqrt{c})^{\frac{3}{4}} (-4 - 3 \sqrt{c} + c) \ln(-(1 - \sqrt{c})^{\frac{1}{4}})                +            4 i (1 + \sqrt{c})^{\frac{3}{4}}         \\  
  &  &      (-4 - 3 \sqrt{c} + c) \ln(-i (1 - \sqrt{c})^{\frac{1}{4}}) +             16 i (1 + \sqrt{c})^{\frac{3}{4}} \ln(i (1 - \sqrt{c})^{\frac{1}{4}}) +                   12 i (1 + \sqrt{c})^{\frac{3}{4}} \sqrt{c} \ln(i (1 - \sqrt{c})^{\frac{1}{4}})         \\         
  &  &    -      4 i (1 + \sqrt{c})^{\frac{3}{4}} c \ln(i (1 - \sqrt{c})^{\frac{1}{4}}) +                     4 (1 + \sqrt{c})^{\frac{3}{4}} \ln(1 - \sqrt{c}) +          3 (1 + \sqrt{c})^{\frac{3}{4}} \sqrt{c} \ln(1 - \sqrt{c})         - (1 + \sqrt{c})^{\frac{3}{4}} c        \\  
  &  &     \ln(1 - \sqrt{c}) +           16 (1 - \sqrt{c})^{\frac{3}{4}} \ln(-(1 + \sqrt{c})^{\frac{1}{4}}) -      12 (1 - \sqrt{c})^{\frac{3}{4}} \sqrt{c} \ln(-(1 + \sqrt[c])^{\frac{1}{4}})           -                 4 (1 - \sqrt{c})^{\frac{3}{4}} c            \\  
  &  &     \ln(-(1 + \sqrt{c})^{\frac{1}{4}}) +           16 i (1 - \sqrt{c})^{\frac{3}{4}} \ln(-i (1 + \sqrt{c})^{\frac{1}{4}}) -     12 i (1 - \sqrt{c})^{\frac{3}{4}} \sqrt{c} \ln(-i (1 + \sqrt[c])^{\frac{1}{4}}) -           4 i (1 - \sqrt{c})^{\frac{3}{4}}        \\  
  &  &     c \ln(-i (1 + \sqrt{c})^{\frac{1}{4}}) -     16 i (1 - \sqrt{c})^{\frac{3}{4}} \ln(i (1 + \sqrt{c})^{\frac{1}{4}}) +      12 i (1 - \sqrt{c})^{\frac{3}{4}} \sqrt{c}       \ln(i (1 + \sqrt{c})^{\frac{1}{4}}) +            4 i (1 - \sqrt{c})^{\frac{3}{4}} c      \\  
  &  &       \ln(i (1 + \sqrt{c})^{\frac{1}{4}}) -     4 (1 - \sqrt{c})^{\frac{3}{4}} \ln(1 + \sqrt{c}) +     3 (1 - \sqrt{c})^{\frac{3}{4}} \sqrt{c} \ln(1 + \sqrt{c}) + (1 - \sqrt[c])^{\frac{3}{4}} c \ln(1 + \sqrt{c})))    
\end{eqnarray*}
\end{theorem}%

\begin{proof}

  Consider: $\sum_{n=2}^{\infty}      \frac{((1 - x^4)^{2 n}) n}{c^n}    =    \frac{(-1 + x^4)^{4} (-1 + 2 c + 2 x^{4} - x^{8})}{c (-1 + c + 2 x^{4} -    x^{8})^2} $.

\end{proof}

 Such elaborate generalized hypergeometric series will perhaps serve as a decent conclusion to this paper.

  In summary, the author has presented somewhat straightforward techniques to evaluate multifarious hypergeometric series, and has evaluated a variety of
 hypergeometric series (some of which could possibly be original), and series which cannot be expressed by a single hypergeometric series (some of which could
 possibly be original).

  Much thanks to my parents for all their support.

{\tt jmaxwellcampbell@gmail.com}


\begin{thebibliography}{9}

\bibitem{almkv} G. Almkvist, C. Krattenthaler \& J. Petersson. Some New Formulas for $\pi$. Source: Experiment. Math Volume 12, Number 4 (2003), 441-456.   \href{http://www.projecteuclid.org}{\tt http://www.projecteuclid.org}    
\bibitem{ramanujan1} B.C. Berndt. Ramanujan's Notebooks, Part I. Springer-Verlag, New York, 1985, p. 106-107.
\bibitem{bo} J. Borwein, D. Bailey. Mathematics by Experiment: Plausible Reasoning in the 21st Century. Second Edition, AK Peters. Wellesley, Massachusetts, 2008. p.2, 10.
\bibitem{bbg} J. Borwein, D. Bailey \& R. Girgensohn. Experimentation in Mathematics: Computational Paths to Discovery. AK Peters, 2004, Natick, MA. p. 126.
\bibitem{borweinchamber} J. Borwein \& M. Chamberland. Integer Powers of Arcsin. Hindawi Publishing Corporation, International Journal of Mathematics and Mathematical Sciences. Volume 2007, Article ID 19381, p. 1-10. \href{http://www.emis.de}{\tt http://www.emis.de}    
\bibitem{jchoi} J. Choi \& H.M. Srivastava. A Reductible Case of Double Hypergeometric Series Involving The Riemann $\zeta$ -Function. Bull. Korean Math. Soc. 33, 1996, No. 1, p. 107-110. \href{http://www.mathnet.or.kr}{\tt http://www.mathnet.or.kr}       
\bibitem{ira}  I.M. Gessel. Finding Identities with the WZ Method. J. Symbolic Comput. 20 (1995), 537�566.      \href{http://people.brandeis.edu/~gessel/homepage/papers/WZ.pdf}{\tt http://people.brandeis.edu/~gessel/homepage/papers/WZ.pdf}             
\bibitem{gradsh} I.S. Gradshteyn \& I.M. Ryzhik. Table of Integrals, Series, and Products. Seventh Edition. Ed. A. Jeffrey \& D. Zwillinger. Academic Press, 2007, Burlington, MA. p. 61, 586.
\bibitem{jolley} L. B. W. Jolley. Summation of Series. Second Revised Edition. Dover Publications, Inc., New York, 1961. p. 156-159.
\bibitem{ab} M. Petkovsek, H.S. Wilf, D. Zeilberger. A=B. Wellesley, MA: A K Peters, p. 18, 1996.     \href{http://www.cis.upenn.edu/~wilf/AeqB.html}{http://www.cis.upenn.edu/~wilf/AeqB.html}                    
\bibitem{sofo} A. Sofo. Computation Techniques for the Summation of Series. Kluwer Academic / Plenum Publishers, 2003, p. 123.
\bibitem{spiegel} M. Spiegel. Theory and Problems of Complex Variables with an Introduction to Conformal Mapping and its Applications. Schaum Publishing Company, New York, 1964, p. 306.
\bibitem{weiss} E.W. Weisstein. Polylogarithm. From  \emph{MathWorld} - A Wolfram Web Resource. \href{http://mathworld.wolfram.com/Polylogarithm.html}{\tt http://mathworld.wolfram.com/Polylogarithm.html}       

\end{thebibliography}
\end{document}